\documentclass[reqno, 11pt]{amsart}
\usepackage{amssymb,amsmath,amsfonts}
\usepackage[margin=1in]{geometry}
\usepackage{dsfont}
\usepackage{color}
\usepackage{graphicx}
\usepackage{latexsym}
\usepackage{amsmath}
\usepackage{amssymb}
\usepackage{graphics}
\usepackage[dvips]{epsfig}
\usepackage{mathrsfs}
\usepackage{mathtools}
\usepackage{leqno}
\usepackage{stmaryrd,cite}

\usepackage{cases}
\usepackage{enumerate}
\usepackage[usenames,dvipsnames,svgnames]{xcolor}
\definecolor{refkey}{rgb}{1,0,0.5}
\definecolor{labelkey}{rgb}{0,0.4,1}
\usepackage{todonotes}
\usepackage{marginnote}

\presetkeys{todonotes}{fancyline, color=green!40}{}

\makeatletter
\renewcommand{\@todonotes@drawMarginNoteWithLine}{%
	\begin{tikzpicture}[remember picture, overlay, baseline=-0.75ex]%
	\node [coordinate] (inText) {};%
	\end{tikzpicture}%
	\marginnote[{
		\@todonotes@drawMarginNote%
		\@todonotes@drawLineToLeftMargin%
	}]{
		\@todonotes@drawMarginNote%
		\@todonotes@drawLineToRightMargin%
	}%
}
\makeatother
\usepackage{listings}
\usepackage{ifpdf}
\ifpdf
\usepackage[CJKbookmarks=true,
         hyperindex=true,
         pdfstartview=FitH,
         bookmarksnumbered=true,
         bookmarksopen=true,
         colorlinks=true,
         citecolor=blue,
         linkcolor=blue,
         urlcolor=blue,
         pdfborder=001,
         pdfauthor={},
         pdftitle={},
         pdfkeywords={},
         ]{hyperref}
\else
\usepackage[
         hypertex,
         hyperindex,
         linkcolor=blue,
         unicode,
         citecolor=blue%
         ]{hyperref}
\fi
\usepackage{cleveref}

\allowdisplaybreaks

\numberwithin{equation}{section}
\newtheorem{theorem}{Theorem}[section]

\newtheorem{lemma}[theorem]{Lemma}

\newtheorem{remark}[theorem]{Remark}

\newcommand{\be}{\begin{equation}}
\newcommand{\ee}{\end{equation}}

\newcommand{\bee}{\begin{equation*}}
\newcommand{\eee}{\end{equation*}}

\newcommand{\bse}{\begin{subequations}}
\newcommand{\ese}{\end{subequations}}

\newcommand{\bs}{\begin{split}}
\newcommand{\es}{\end{split}}

\begin{document}

\author{Hairong Liu$^{1}$}\thanks{$^{1}$School of Mathematics and Statistics, Nanjing University of Science and Technology,  Nanjing 210094 P.R.China.
E-mail : hrliu@njust.edu.cn}

\author{Long Tian$^{2}$}\thanks{$^{2}$School of Mathematics and Statistics, Nanjing University of Science and Technology,  Nanjing 210094 P.R.China.
E-mail: tianlong19850812@163.com}

\author{Xiaoping Yang$^{3}$}\thanks{$^{3}$School of Mathematics,  Nanjing University, Nanjing 210093, P.R. China.
E-mail: xpyang@nju.edu.cn}

\title[] {Quantitative unique continuation property for solutions to a bi-Laplacian equation with a potential}

\begin{abstract}
In this paper, we focus on the quantitative unique continuation property of solutions to
\begin{equation*}
\Delta^2u=Vu,
\end{equation*}
where $V\in W^{1,\infty}$.
We show that
the maximal  vanishing order of the solutions
is not large than
\begin{equation*}
C\left(\|V\|^{\frac{1}{4}}_{L^{\infty}}+\|\nabla V\|_{L^{\infty}}+1\right).
\end{equation*}
 Our key argument is to lift the original equation to that with a positive potential, then decompose the resulted fourth-order equation into a special system of two second-order equations. Based on the special system, we define a variant frequency function with weights and derive its almost monotonicity to establishing some doubling inequalities with explicit dependence on the Sobolev norm of the potential function.


\end{abstract}

\maketitle

\section{Introduction and main theorem}
This paper is devoted to investigating the quantitative unique continuation property  for solutions to a bi-Laplacian equation with a potential.
A function $u\in L^2$ is said to have vanishing order $k\geq 0$ at some point $x_0\in\mathbb{R}^{n}$ if
\begin{align}\label{def}
 \frac{\int_{B_{r}(x_0)}u^{2}dx}{r^{n+2k}}=O(1),
\end{align}
and  vanish to infinite order at point $x_0$ if
\begin{eqnarray*}
\int_{B_r(x_0)}u^2=o(r^{n+2k})
\quad \mbox{for any integer} \quad k.
\end{eqnarray*}
A differential operator $L$ is said to have the strong unique continuation property in a connected domain $\Omega$ if the only solution of $Lu=0$
which vanishes to infinite order at a point $x_0\in\Omega$ is $u\equiv0$. If the strong unique continuation property holds for $L$, that means the nontrivial solutions of $Lu=0$ do not vanish of infinite order.
It is interesting to characterize the vanishing orders  of solutions by the coefficient functions appeared in the equations $Lu=0$, this is called the quantitative unique continuation property.

In the past decades the strong unique continuation property for solutions to various kinds of
PDEs has attracted a large number of researchers and induced
 many interesting and intensive results.
 The results  in this aspect for the second order operator
\begin{align}\label{sch}
-\Delta u=Vu
\end{align}
go back to the work of Carleman \cite{c1939}, who solved the uniqueness problem in $\mathbb{R}^{2}$ with  bounded potentials.  Cordes \cite{c1956} and Aronszajn \cite{a1957} extended the strong unique continuation property to second-order equations in $\mathbb{R}^{n}$.
One of subsequent significant developments
in this direction  is due to Jerison and Kenig \cite{jk1985} who studied
the strong unique continuation property for (\ref{sch}) with $V\in L_{loc}^{n/2}(\mathbb{R}^{n})$.
On the other hand, Garofalo and Lin \cite{G-L1986,G-L1987} presented a geometric-variational approach to the strong unique continuation by defining a kind of  frequency functions. Their method is based on establishing one doubling estimate which in turn depends on the monotonicity property of the frequency functions.
It is worth pointing out that the frequency function was first introduced by Almgren \cite{alm} for harmonic functions.
At present, the Carleman estimates and frequency functions  are two principal ways to obtain  strong and/or quantitative unique continuation properties. There is a large amount of work on strong unique continuation  for second order elliptic operators (cf. \cite{DLW2021, h2001} and references therein).

It is well known that all zeros of nontrivial solutions of second order
linear equations on smooth compact Riemannian manifolds are of finite order.
For classic eigenfunctions on a compact smooth Riemannian manifold $M$,
\begin{align*}
-\Delta_{g}\varphi_{\lambda}=\lambda\varphi_{\lambda}\quad \mbox{in}\quad M,
\end{align*}
Donnelly, Fefferman \cite{D1988} and Lin \cite{lin1991} showed that the maximal vanishing order of  eigenfunctions $\varphi_{\lambda}$ is
less than $C\sqrt{\lambda}$, here $C$ only depends on the manifold. Kukavica  \cite{K1998} studied
the vanishing order of solutions to the Schr\"{o}dinger equation:
\begin{equation}\label{S}
-\Delta u=V(x)u.
\end{equation}
Kukavica showed  that the vanishing order of solutions is less than
$C\Big(1+\|V^{-}\|_{L^{\infty}}^{1/2}+\mbox{osc}V+\|\nabla V\|_{L^{\infty}}\Big)$ provided $V\in W^{1,\infty}$ by using the frequency function argument.
 However, this upper bound is not sharp compared to the case  $V(x)=\lambda$.
Bakri \cite{B2012} and Zhu \cite{zhu2016} improved the  upper bounded of vanishing order to $\left(1+\sqrt{\|V\|_{W^{1,\infty}}}\right)$ by different methods.
On the other hand, if $V(x)\in L^{\infty}$, Bourgain and Kenig \cite{BK}
showed that the  upper bounded of vanishing order is not large than $C\|V\|^{\frac{2}{3}}_{L^{\infty}}$. Moreover, Kenig \cite{K} also pointed out that the exponent $\frac{2}{3}$ of $\|V\|_{L^{\infty}}$ is sharp for complex valued $V(x)$ thanks to Meshkov's example in \cite{M}.
 In \cite{K}  Kenig  asked if the vanishing order could be improved to $\sqrt{\|V\|_{L^{\infty}}}$, matching that of the eigenfunctions, in the real-valued setting. We also point out the more recent works \cite{D2020,DZ,KS2016} for the quantitative unique continuation properties
of solutions to second order elliptic equations with singular lower order terms.

 The  unique continuation property for higher order elliptic operators is more complex than  second order
operators. Alinhac \cite{a} constructed a strong uniqueness counterexample for differential operators $P$ of any order in $\mathbb{R}^2$, under the condition that $P$ has two simple, nonconjugate complex characteristics. The strong unique continuation property for  the m-th powers of a Laplacian operator with singular coefficients has been intensively investigated, see e.g. \cite{cg1999, ck2010, K2012,linc2007,liu2022}.
However, there are relatively fewer results
about the  quantitative unique continuation property for solutions to higher order elliptic equations.
 Kukavica \cite{K1995} considered eigenfunctions of a $2m$-order regular  elliptic problem on a bounded domain with real-analytic boundary,  showed the maximal order of vanishing corresponding to eigenfunctions $u_{\lambda}$  is  $\lambda^{1/{2m}}$.
 For the upper bounds of vanishing order and estimates on Hausdorff measure of nodal sets corresponding to eigenfunctions  of the bi-Laplace operators  with  different boundary conditions, the reader is referred to recent works
\cite{LIN2022,TY2022,TY2023}. On the other hand,
for the higher order elliptic equations with a potential
\begin{align*}
\Delta^{m}u+V(x)u=0,
\end{align*}
Zhu  \cite{zhu2016}  proved that  the vanishing order of $u$ is less than $\|V\|_{L^{\infty}}$ for $n\geq 4m$ by
a variant of frequency function. Zhu \cite{zhu2019} improved the vanishing order of solutions to
\begin{align*}
\Delta^2 u=V(x)u
\end{align*}
to $C\|V\|^{\frac{1}{3}}_{L^{\infty}}$. See \cite{linc2011,ZHU2018} for more results about
the quantitative unique continuation properties of solutions to higher order elliptic
equations.

Based on the above results on the maximal order of vanishing corresponding to eigenfunctions of bi-Laplacian \cite{K1995,LIN2022,TY2022,TY2023},
 a natural problem is that whether the order of vanishing can be reduced to $\|V\|^{\frac{1}{4}}_{L^{\infty}}$
 for the solutions to $\Delta^{2}u=V(x)u$.

 The goal of the paper is to  continue to investigate   the above problem in some sense and try to improve the upper bound of   maximal vanishing order of solutions to the following bi-Laplacian with a  potential
\begin{equation*}
\Delta^2u=V(x)u\quad \mbox{in}\quad  B_{1}(0),
\end{equation*}
under the assumption $V\in W^{1,\infty}$, where $B_1(0)$ is a ball centered at origin with radius 1 in $\mathbb{R}^n$ ($n\geq2$). Precisely,
the main result of this paper is the following theorem.
\begin{theorem}\label{thm}
Let $u$ be a solution of
\begin{equation} \label{equ1}
\Delta^2u=V(x)u \quad \mbox{in} \quad B_{1}(0).
\end{equation}
Assume that $V\in W^{1,\infty}$, $\|u\|_{L^{\infty}(B_{1}(0))}\leq C_0$ and $u(0)\geq 1$.
Then the vanishing order of  $u$ at any point $x\in B_1(0)$ is less than
\begin{equation*}
C\left(\|V\|^{\frac{1}{4}}_{L^{\infty}(B_1(0))}+\|\nabla V\|_{L^{\infty}(B_1(0))}+1\right),
\end{equation*}
where $C$ depends only on $n$ and $C_0$.
\end{theorem}

\begin{remark}It is worth mentioning that the exponent $\frac{1}{4}$ of  $\|V\|_{L^{\infty}}$ is optimal based on the analogues  of the eigenvalue problem to bi-Laplacian \cite{K1995,LIN2022,TY2022}.
\end{remark}

Our key argument is to lift the original equation to that with a signed potential, and define a variant frequency function with weights and derive its  almost monotonicity to establishing some  doubling  inequalities.
The rest of the paper is organized as follows. Section 2 is devoted to establish some doubling inequalities with explicit dependence on the Sobolev norm of the potential function $V$. We first  lift the  equation (\ref{equ1}) to that with a  positive potential in $\mathbb{R}^{n+1}$ by introducing an auxiliary function, and then  decompose the resulted fourth-order equation  into a special system of two second order equations (see (\ref{equ}) below). Based on the system, we define a variant frequency function with weights and derive its  monotonicity to establishing some  doubling  inequalities. In Section 3, we prove Theorem \ref{thm}. Precisely, by virtue of doubling inequalities and some elliptic interior estimates, we can establish that quantitative relationship between  the  vanishing order  and the frequency $N(r)$.

\section{Frequency function and the doubling estimates}

In this section, we first transform the  equation (\ref{equ1}) by lifting into  a fourth order equation with a signed potential term, then decompose the resulted fourth order equation  into a suitable system of two second order equations. For the systems, we define a frequency function $N(r)$  and establish its almost monotonicity. Finally, with help of the  almost  monotonicity of $N(r)$, we derive some doubling estimates and a changing center
property of $N(r)$.

Let $u=u(x)$ be a solution to (\ref{equ1}) and  the transformation
\begin{align*}
\widetilde{u}(x,t)=u(x)e^{\sqrt{\lambda}t},
\end{align*}
constant $\lambda>0$ to be determined later.
Then it is easy to see that $\widetilde{u}(x,t)$  satisfies the equation
\begin{align}\label{equ2-1}
\Delta_{x,t}^2\widetilde{u}(x,t)=2\lambda\Delta_{x,t} \widetilde{u}(x,t)+(V(x)-\lambda^2)\widetilde{u}(x,t),
\end{align}
where $\Delta_{x,t}$ denotes the Laplacian with respect to $(x,t)$ in $\mathbb{R}^{n+1}$, i.e. $\Delta_{x,t}=\sum_{i=1}^{n}\partial_{x_ix_i}+\partial_{tt}$.
Noting that the vanishing order of $u(x)$ and $\widetilde{u}(x,t)$ is the same, from now  on, we only need to  study equation (\ref{equ2-1}) and $\widetilde{u}(x,t)$. Let $z=(x,t)$,  $z_0=(x,0)$ and $B_r(z_0)$ denote the ball centered at $z_0$ with radius $r$  in $\mathbb{R}^{n+1}$.  Moreover, For simplicity, we'll use $u(x,t)$ instead of $\widetilde{u}(x,t)$, and $\Delta$ instead of $\Delta_{x,t}$, that is
\begin{align}\label{equ2}
\Delta^2u(x,t)=2\lambda\Delta u(x,t)+(V(x)-\lambda^2)u (x,t)\quad \mbox{in} \quad B_{1}(0,0).
\end{align}
We decompose equation (\ref{equ2}) into the following system of two second-order equations:
\begin{equation}\label{equ}
\left\{\begin{array}{lll}
w(x,t)=\Delta u(x,t)-\frac{1}{2}\lambda u(x,t),\\[2mm]
\Delta w(x,t)-\frac{3}{2}\lambda w(x,t)=\left(V(x)-\frac{1}{4}\lambda^2\right)u(x,t).
\end{array}\right.
\end{equation}
Denoting  $\|V\|_{L^{\infty}}\equiv\|V(x)\|_{L^{\infty}(B_1(0))}$ and $\|\nabla V\|_{L^{\infty}}\equiv\|\nabla V(x)\|_{L^{\infty}(B_1(0))}.$
Taking $\lambda=2\|V\|^{\frac{1}{2}}_{L^{\infty}}$, so
\begin{align}\label{V}
\sqrt{\lambda}=\sqrt{2}\|V\|^{\frac{1}{4}}_{L^{\infty}}, \quad \left\|V-\frac{1}{4}\lambda^2\right\|_{L^{\infty}}\leq \left\|\nabla V\right\|_{L^{\infty}}.
\end{align}
For any $B_{r}(z_0)\subset B_1(0,0)$, define
\begin{equation}\label{H}
H(z_0,r)=\int_{B_r(z_0)}(u^2+w^2)(r^2-|z-z_0|^2)^{\alpha}dz,
\end{equation}
and
\begin{align}\label{I}
I(z_0,r)=2(\alpha+1)\int_{B_r(z_0)}\left(u\nabla u+w\nabla w\right)\cdot(z-z_0) (r^2-|z-z_0|^2)^{\alpha}dz,
\end{align}
the constant $\alpha>0$ will be determined later on.
The function in (\ref{H}) is introduced by Kukavica \cite{K2000} for Ginzburg-Landau equations.

By the divergence theorem and the equations in (\ref{equ}), the energy $I(z_0,r)$ can be rewritten as
\begin{align}\label{I-0}
I(z_0,r)&=\int_{B_r(z_0)}|\nabla u|^2(r^2-|z-z_0|^2)^{\alpha+1}dz+\int_{B_r(z_0)}|\nabla w|^2(r^2-|z-z_0|^2)^{\alpha+1}dz\nonumber\\[2mm]
&+\frac{1}{2}\lambda\int_{B_r(z_0)} u^2(r^2-|z-z_0|^2)^{\alpha+1}dz+\frac{3}{2}\lambda \int_{B_r(z_0)}w^2(r^2-|z-z_0|^2)^{\alpha+1}dz\nonumber\\[2mm]
&+\int_{B_r(z_0)}\left(1+V-\frac{1}{4}\lambda^2\right)uw(r^2-|z-z_0|^2)^{\alpha+1}dz\nonumber\\[2mm]
&:=I_1(r)+I_2(r)+I_3(r)+I_4(r)+I_5(r).
\end{align}
We emphasize  that the term
\begin{align*}
\lambda\int_{B_r(z_0)} u^2(r^2-|z-z_0|^2)^{\alpha+1}dz
\end{align*}
appears in $I(z_0,r)$ mainly because of the special decomposition (\ref{equ}), 
which is crucial to control the term $R_5^4$ in  $I'(r)$ (see (\ref{I5}) below).

The frequency function is defined as
 \begin{align}\label{N}
 N(z_0,r)=\frac{I(z_0,r)}{H(z_0,r)}.
 \end{align}
Here and in what follows, without causing confusion, the center of the ball is omitted in the notation $H(z_0,r), I(z_0,r)$ and $N(z_0,r)$.
Next, we establish an almost monotonicity property of the frequency function, which is the key tool to  prove our main theorem. Precisely,
\begin{lemma}\label{keylemma}
Let $u$ be  a solution of (\ref{equ2}).
For any $z_0\in B_{1}(0,0)$ with $B_{r}(z_0)\subset B_{1}(0,0)$,   then
there exists a constant $C$ depending only on $n$ such that
\begin{equation*}
e^{Cr}\Big(N(z_0,r)+\|\nabla V\|_{L^{\infty}}+1\Big)
\end{equation*}
is  nondecreasing  with respect to $r\in (0,1)$.
\end{lemma}
\begin{proof}
In order to obtain  the almost  monotonicity of $N(r)$,
we shall calculate and estimate $H'(r)$ and $I'(r)$, respectively.

{\it Step 1.}\ Calculate $H'(r)$.  Taking the derivative for $H(r)$ with respect to $r$, one has
\begin{align}\label{H'}
H'(r)&=2\alpha r\int_{B_r(z_0)}\left(u^2+w^2\right)(r^2-|z-z_0|^2)^{\alpha-1}dz\nonumber\\[2mm]
&=\frac{2\alpha}{r}\int_{B_r(z_0)}\left(u^2+w^2\right)(r^2-|z-z_0|^2)^{\alpha}dz\nonumber\\[2mm]
&+\frac{2\alpha}{r}\int_{B_r(z_0)}\left(u^2+w^2\right)|z-z_0|^2(r^2-|z-z_0|^2)^{\alpha-1}dz\nonumber\\[2mm]
&:=\frac{2\alpha}{r}H(r)+K_1.
\end{align}
Using the  following identity,
\begin{align*}
|z-z_0|^2(r^2-|z-z_0|^2)^{\alpha-1}=-\frac{1}{2\alpha}(z-z_0)\cdot\nabla(r^2-|z-z_0|^2)^{\alpha},
\end{align*}
and  integrating by parts,  it holds
\begin{align}\label{K}
K_1&=-\frac{1}{r}\int_{B_r(z_0)}\left(u^2+w^2\right)(z-z_0)\cdot\nabla(r^2-|z-z_0|^2)^{\alpha}dz\nonumber\\[2mm]
&=\frac{1}{r}\int_{B_r(z_0)}\mbox{div}\left((u^2+w^2)(z-z_0)\right)(r^2-|z-z_0|^2)^{\alpha}dz\nonumber\\[2mm]
&=\frac{n+1}{r}\int_{B_r(z_0)}(u^2+w^2)(r^2-|z-z_0|^2)^{\alpha}dz\nonumber\\[2mm]
&+\frac{2}{r}\int_{B_r(z_0)}(u\nabla u+w\nabla w)\cdot(z-z_0)(r^2-|z-z_0|^2)^{\alpha}dz.
\end{align}
Substituting (\ref{K}) into (\ref{H'}), and recalling the definition of $I(r)$ in (\ref{I}), one gets
\begin{align}\label{fin-1}
H'(r)=\frac{2\alpha+n+1}{r}H(r)+\frac{1}{(\alpha+1)r}I(r).
\end{align}
{\it Step 2.} \ Calculate $I'(r)$.  Recalling the definition of $I(r)$ in (\ref{I-0}),  we compute $I_i'(r)$($i=1,2,\cdots,5$) one by one.
\begin{align}\label{0}
I_1'(r)&=2(\alpha+1)r\int_{B_r(z_0)}|\nabla u|^2(r^2-|z-z_0|^2)^{\alpha}dz\nonumber\\[2mm]
&=\frac{2(\alpha+1)}{r}\int_{B_r(z_0)}|\nabla u|^2(r^2-|z-z_0|^2)^{\alpha+1}dz\nonumber\\[2mm]
&+\frac{2(\alpha+1)}{r}\int_{B_r(z_0)}|\nabla u|^2|z-z_0|^2(r^2-|z-z_0|^2)^{\alpha}dz.
\end{align}
Using the  following identity,
\begin{align}\label{identity}
|z-z_0|^2(r^2-|z-z_0|^2)^{\alpha}=-\frac{1}{2(\alpha+1)}(z-z_0)\cdot\nabla(r^2-|z-z_0|^2)^{\alpha+1},
\end{align}
and  integrating by parts, then the last term of (\ref{0}) becomes
\begin{align}\label{0-1}
&\quad \frac{2(\alpha+1)}{r}\int_{B_r(z_0)}|\nabla u|^2(r^2-|z-z_0|^2)^{\alpha}|z-z_0|^2dz\nonumber\\[2mm]
&=\frac{1}{r}\int_{B_r(z_0)}\mbox{div}\left(|\nabla u|^2(z-z_0)\right)(r^2-|z-z_0|^2)^{\alpha+1}dz\nonumber\\[2mm]
&=\frac{n+1}{r}\int_{B_r(z_0)}|\nabla u|^2(r^2-|z-z_0|^2)^{\alpha+1}dz+\frac{2}{r}\int_{B_r(z_0)}\partial_{i}u\partial_{ij} u(z-z_0)_{j}(r^2-|z-z_0|^2)^{\alpha+1}dz
\nonumber\\[2mm]
&=\frac{n+1}{r}\int_{B_r(z_0)}|\nabla u|^2(r^2-|z-z_0|^2)^{\alpha+1}dz-\frac{2}{r}\int_{B_r(z_0)}\partial_{j} u\partial_{i}\left(\partial_{i}u(z-z_0)_{j}(r^2-|z-z_0|^2)^{\alpha+1}\right)dz
\nonumber\\[2mm]
&=\frac{n-1}{r}\int_{B_r(z_0)}|\nabla u|^2(r^2-|z-z_0|^2)^{\alpha+1}dz+\frac{4(\alpha+1)}{r}\int_{B_r(z_0)}\Big(\nabla u\cdot(z-z_0)\Big)^2(r^2-|z-z_0|^2)^{\alpha}dz\nonumber\\[2mm]
&\quad -\frac{2}{r}\int_{B_r(z_0)}\Delta u\nabla u\cdot(z-z_0) (r^2-|z-z_0|^2)^{\alpha+1}dz\nonumber\\[2mm]
&=\frac{n-1}{r}\int_{B_r(z_0)}|\nabla u|^2(r^2-|z-z_0|^2)^{\alpha+1}dz+\frac{4(\alpha+1)}{r}\int_{B_r(z_0)}\Big(\nabla u\cdot(z-z_0)\Big)^2(r^2-|z-z_0|^2)^{\alpha}dz\nonumber\\[2mm]
&\quad -\frac{1}{(\alpha+2)r}\int_{B_r(z_0)}\mbox{div}\left(\Delta u\nabla u\right)(r^2-|z-z_0|^2)^{\alpha+2}dz,
\end{align}
where we have used the following fact and integration  by parts again in the last equality,
\begin{align}\label{id2}
(z-z_0) (r^2-|z-z_0|^2)^{\alpha+1}=-\frac{1}{2(\alpha+2)}\nabla (r^2-|z-z_0|^2)^{\alpha+2}.
\end{align}
Putting (\ref{0-1}) into (\ref{0}) yields
\begin{align}\label{0-2}
I_1'(r)
&=\frac{2\alpha+n+1}{r}I_1(r)+\frac{4(\alpha+1)}{r}\int_{B_r(z_0)}\Big(\nabla u\cdot(z-z_0)\Big)^2(r^2-|z-z_0|^2)^{\alpha}dz\nonumber\\[2mm]
&-\frac{1}{(\alpha+2)r}\int_{B_r(z_0)}\mbox{div}\left(\Delta u\nabla u\right)(r^2-|z-z_0|^2)^{\alpha+2}dz.
\end{align}
Furthermore, for the last term on the right hand side of (\ref{0-2}), by using  the first equation in (\ref{equ}),   it follows that
\begin{align}\label{I1}
I_1'(r)
&=\frac{2\alpha+n+1}{r}I_1(r)+\frac{4(\alpha+1)}{r}\int_{B_r(z_0)}\Big(\nabla u\cdot(z-z_0)\Big)^2(r^2-|z-z_0|^2)^{\alpha}dz\nonumber\\[2mm]
&-\frac{1}{(\alpha+2)r}\int_{B_r(z_0)}w^2(r^2-|z-z_0|^2)^{\alpha+2}dz\nonumber\\[2mm]
&-\frac{1}{(\alpha+2)r}\int_{B_r(z_0)}\lambda u w(r^2-|z-z_0|^2)^{\alpha+2}dz\nonumber\\[2mm]
&-\frac{1}{(\alpha+2)r}\int_{B_r(z_0)}\nabla u\cdot\nabla w(r^2-|z-z_0|^2)^{\alpha+2}dz\nonumber\\[2mm]
&-\frac{1}{2(\alpha+2)r}\int_{B_r(z_0)}\lambda |\nabla u|^2(r^2-|z-z_0|^2)^{\alpha+2}dz\nonumber\\[2mm]
&-\frac{1}{4(\alpha+2)r}\int_{B_r(z_0)}\lambda^2u^2(r^2-|z-z_0|^2)^{\alpha+2}dz\nonumber\\[2mm]
&:=\frac{2\alpha+n+1}{r}I_1(r)+\frac{4(\alpha+1)}{r}\int_{B_r(z_0)}\Big(\nabla u\cdot(z-z_0)\Big)^2(r^2-|z-z_0|^2)^{\alpha}dz\nonumber\\[2mm]
&+R_{1}^{1}+R_{1}^{2}+R_{1}^{3}+R_{1}^{4}+R_{1}^{5}.
\end{align}
In a similar way, we have
\begin{align*}
I_2'(r)
&=\frac{2\alpha+n+1}{r}I_2(r)+\frac{4(\alpha+1)}{r}\int_{B_r(z_0)}\Big(\nabla w\cdot(z-z_0)\Big)^2(r^2-|z-z_0|^2)^{\alpha}dz\nonumber\\[2mm]
&-\frac{1}{(\alpha+2)r}\int_{B_r(z_0)}\mbox{div}\left(\Delta w\nabla w\right)(r^2-|z-z_0|^2)^{\alpha+2}dz.
\end{align*}
Moreover, by using the second equation  in (\ref{equ}),   it yields
\begin{align}
I_2'(r)
&=\frac{2\alpha+n+1}{r}I_2(r)+\frac{4(\alpha+1)}{r}\int_{B_r(z_0)}\Big(\nabla w\cdot(z-z_0)\Big)^2(r^2-|z-z_0|^2)^{\alpha}dz\nonumber\\[2mm]
&-\frac{1}{(\alpha+2)r}\int_{B_r(z_0)}\left(V-\frac{1}{4}\lambda^2\right)^2u^2(r^2-|z-z_0|^2)^{\alpha+2}dz\nonumber\\[2mm]
&-\frac{3}{(\alpha+2)r}\int_{B_r(z_0)}\lambda\left(V-\frac{1}{4}\lambda^2\right)uw(r^2-|z-z_0|^2)^{\alpha+2}dz\nonumber\\[2mm]
&-\frac{1}{(\alpha+2)r}\int_{B_r(z_0)}u\nabla V\cdot\nabla w(r^2-|z-z_0|^2)^{\alpha+2}dz\nonumber\\[2mm]
&-\frac{1}{(\alpha+2)r}\int_{B_r(z_0)}\left(V-\frac{1}{4}\lambda^2\right)\nabla u\cdot\nabla w(r^2-|z-z_0|^2)^{\alpha+2}dz\nonumber\\[2mm]
&-\frac{3}{2(\alpha+2)r}\int_{B_r(z_0)}\lambda |\nabla w|^2(r^2-|z-z_0|^2)^{\alpha+2}dz\nonumber\\[2mm]
&-\frac{9}{4(\alpha+2)r}\int_{B_r(z_0)}\lambda^2 w^2(r^2-|z-z_0|^2)^{\alpha+2}dz\nonumber\\[2mm]
&:=\frac{2\alpha+n+1}{r}I_2(r)+\frac{4(\alpha+1)}{r}\int_{B_r(z_0)}\Big(\nabla w\cdot(z-z_0)\Big)^2(r^2-|z-z_0|^2)^{\alpha}dz\nonumber\\[2mm]
&+R_{2}^{1}+R_{2}^{2}+R_{2}^{3}+R_{2}^{4}+R_{2}^{5}+R_{2}^{6}.
\end{align}
For
\begin{align*}
I_3(r)=\frac{1}{2}\lambda\int_{B_r(z_0)} u^2(r^2-|z-z_0|^2)^{\alpha+1}dz,
\end{align*}
we denote $I_3(r)=\frac{1}{2}\lambda \widetilde{I_3}(r)$.  Taking the derivative for $\widetilde{I_3}(r)$ with respect to $r$, using the identity (\ref{identity}) and integrating by parts, one obtains
\begin{align}\label{I3-1}
\widetilde{I_3}'(r)&=2(\alpha+1)r\int_{B_r(z_0)} u^2(r^2-|z-z_0|^2)^{\alpha}dz\nonumber\\[2mm]
&=\frac{2(\alpha+1)}{r}\int_{B_r(z_0)} u^2(r^2-|z-z_0|^2)^{\alpha+1}dz+\frac{2(\alpha+1)}{r}\int_{B_r(z_0)} u^2(r^2-|z-z_0|^2)^{\alpha}|z-z_0|^2dz\nonumber\\[2mm]
&=\frac{2(\alpha+1)}{r}\int_{B_r(z_0)} u^2(r^2-|z-z_0|^2)^{\alpha+1}dz+\frac{1}{r}\int_{B_r(z_0)}
\mbox{div}\Big(u^2(z-z_0)\Big)(r^2-|z-z_0|^2)^{\alpha+1}dz\nonumber\\[2mm]
&=\frac{2(\alpha+1)+n+1}{r}\widetilde{I_3}(r)+\frac{2}{r}\int_{B_r(z_0)} u\nabla u\cdot(z-z_0)(r^2-|z-z_0|^2)^{\alpha+1}dz.
\end{align}
Furthermore,  by using the identity (\ref{id2}),  integrating by parts and using  the first equation  in (\ref{equ}), the last term on  (\ref{I3-1}) becomes
\begin{align}\label{I3-0}
&\frac{2}{r}\int_{B_r(z_0)} u\nabla u\cdot(z-z_0)(r^2-|z-z_0|^2)^{\alpha+1}dz\nonumber\\[2mm]
&=\frac{1}{(\alpha+2)r}\int_{B_r(z_0)} \mbox{div}\Big(u\nabla u\Big)(r^2-|z-z_0|^2)^{\alpha+2}dz\nonumber\\[2mm]
&=\frac{1}{(\alpha+2)r}\int_{B_r(z_0)}|\nabla u|^2(r^2-|z-z_0|^2)^{\alpha+2}dz
+\frac{1}{2(\alpha+2)r}\lambda\int_{B_r(z_0)}u^2(r^2-|z-z_0|^2)^{\alpha+2}dz\nonumber\\[2mm]
&+\frac{1}{(\alpha+2)r}\int_{B_r(z_0)}uw(r^2-|z-z_0|^2)^{\alpha+2}dz.
\end{align}
Plugging (\ref{I3-0}) back in (\ref{I3-1}), one gets
\begin{align}
I_3'(r)
&=\frac{2\alpha+n+1}{r}I_3(r)+\frac{2}{r}I_{3}(r)\nonumber\\[2mm]
&+\frac{1}{2(\alpha+2)r}\lambda\int_{B_r(z_0)} uw(r^2-|z-z_0|^2)^{\alpha+2}dz\nonumber\\[2mm]
&+\frac{1}{2(\alpha+2)r}\lambda\int_{B_r(z_0)} |\nabla u|^2(r^2-|z-z_0|^2)^{\alpha+2}dz\nonumber\\[2mm]
&+\frac{1}{4(\alpha+2)r}\lambda^2\int_{B_r(z_0)} u^2(r^2-|z-z_0|^2)^{\alpha+2}dz\nonumber\\[2mm]
&:=\frac{2\alpha+n+1}{r}I_3(r)+\frac{2}{r}I_{3}(r)+R_3^1+R_3^2+R_3^3.
\end{align}
Similarly, one has
\begin{align}
I_4'(r)
&=\frac{2\alpha+n+1}{r}I_4(r)+\frac{2}{r}I_{4}(r)\nonumber\\[2mm]
&+\frac{3}{2(\alpha+2)r}\lambda\int_{B_r(z_0)}\left(V-\frac{1}{4}\lambda^2\right) uw(r^2-|z-z_0|^2)^{\alpha+2}dz\nonumber\\[2mm]
&+\frac{3}{2(\alpha+2)r}\lambda\int_{B_r(z_0)} |\nabla w|^2(r^2-|z-z_0|^2)^{\alpha+2}dz\nonumber\\[2mm]
&+\frac{9}{4(\alpha+2)r}\lambda^2\int_{B_r(z_0)} w^2(r^2-|z-z_0|^2)^{\alpha+2}dz\nonumber\\[2mm]
&:=\frac{2\alpha+n+1}{r}I_4(r)+\frac{2}{r}I_{4}(r)+R_4^1+R_4^2+R_4^3,
\end{align}
and
\begin{align}\label{i-5}
I_5'(r)&=\frac{2\alpha+n+1}{r}I_5(r)+\frac{2}{r}I_{5}(r)
+\frac{1}{r}\int_{B_r(z_0)} uw\nabla V\cdot(z-z_0)(r^2-|z-z_0|^2)^{\alpha+1}dz\nonumber\\[2mm]
&+\frac{1}{r}\int_{B_r(z_0)} \left(V-\frac{1}{4}\lambda^2+1\right)(u\nabla w+w\nabla w)\cdot(z-z_0)(r^2-|z-z_0|^2)^{\alpha+1}dz.
\end{align}
Denoting the last term on the right-hand side of (\ref{i-5}) as $\frac{1}{r} K_5$. Similar to (\ref{I3-0}), and by virtue of equation (\ref{equ}),
the term $K_5$ can be expressed as follows
\begin{align}
K_5&=\frac{1}{2(\alpha+2)}\int_{B_r(z_0)}\left(V-\frac{1}{4}\lambda^2\right)^2 u^2(r^2-|z-z_0|^2)^{\alpha+2}dz\nonumber\\[2mm]
&+\frac{1}{2(\alpha+2)}\int_{B_r(z_0)}\left(V-\frac{1}{4}\lambda^2\right) u^2(r^2-|z-z_0|^2)^{\alpha+2}dz\nonumber\\[2mm]
&+\frac{1}{(\alpha+2)}\int_{B_r(z_0)}\lambda\left(V-\frac{1}{4}\lambda^2+1\right) uw(r^2-|z-z_0|^2)^{\alpha+2}dz\nonumber\\[2mm]
&+\frac{1}{2(\alpha+2)}\int_{B_r(z_0)}\left(V-\frac{1}{4}\lambda^2+1\right) w^2(r^2-|z-z_0|^2)^{\alpha+2}dz\nonumber\\[2mm]
&+\frac{1}{(\alpha+2)}\int_{B_r(z_0)}\left(V-\frac{1}{4}\lambda^2+1\right) \nabla u\cdot\nabla w(r^2-|z-z_0|^2)^{\alpha+2}dz\nonumber\\[2mm]
&+\frac{1}{2(\alpha+2)}\int_{B_r(z_0)}(u\nabla w+w\nabla u)\cdot\nabla V(r^2-|z-z_0|^2)^{\alpha+2}dz.\nonumber
\end{align}
Plugging the above back in (\ref{i-5}) yields
\begin{align}\label{I5}
I_5'(r)&=\frac{2\alpha+n+1}{r}I_5(r)+\frac{2}{r}I_{5}(r)\nonumber\\[2mm]
&+\frac{1}{r}\int_{B_r(z_0)} uw\nabla V\cdot(z-z_0)(r^2-|z-z_0|^2)^{\alpha+1}dz\nonumber\\[2mm]
&+\frac{1}{2(\alpha+2)r}\int_{B_r(z_0)}\left(V-\frac{1}{4}\lambda^2\right)^2 u^2(r^2-|z-z_0|^2)^{\alpha+2}dz\nonumber\\[2mm]
&+\frac{1}{2(\alpha+2)r}\int_{B_r(z_0)}\left(V-\frac{1}{4}\lambda^2\right) u^2(r^2-|z-z_0|^2)^{\alpha+2}dz\nonumber\\[2mm]
&+\frac{1}{(\alpha+2)r}\int_{B_r(z_0)}\lambda\left(V-\frac{1}{4}\lambda^2\right) uw(r^2-|z-z_0|^2)^{\alpha+2}dz\nonumber\\[2mm]
&+\frac{1}{(\alpha+2)r}\int_{B_r(z_0)}\lambda uw(r^2-|z-z_0|^2)^{\alpha+2}dz\nonumber\\[2mm]
&+\frac{1}{2(\alpha+2)r}\int_{B_r(z_0)}\left(V-\frac{1}{4}\lambda^2\right) w^2(r^2-|z-z_0|^2)^{\alpha+2}dz\nonumber\\[2mm]
&+\frac{1}{2(\alpha+2)r}\int_{B_r(z_0)} w^2(r^2-|z-z_0|^2)^{\alpha+2}dz\nonumber\\[2mm]
&+\frac{1}{(\alpha+2)r}\int_{B_r(z_0)}\left(V-\frac{1}{4}\lambda^2\right) \nabla u\cdot\nabla w(r^2-|z-z_0|^2)^{\alpha+2}dz\nonumber\\[2mm]
&+\frac{1}{(\alpha+2)r}\int_{B_r(z_0)}\nabla u\cdot\nabla w(r^2-|z-z_0|^2)^{\alpha+2}dz\nonumber\\[2mm]
&+\frac{1}{2(\alpha+2)r}\int_{B_r(z_0)}(u\nabla w+w\nabla u)\cdot\nabla V(r^2-|z-z_0|^2)^{\alpha+2}dz\nonumber\\[2mm]
&:=\frac{2\alpha+n+1}{r}I_5(r)+\frac{2}{r}I_{5}(r)+R_{5}^{1}+R_{5}^{2}+\cdots+R_{5}^{10}.
\end{align}
Summing the $I'_i(r)$ ($i=1,2\cdots,5$) in (\ref{I1})-(\ref{I5}),
and noting that some ``bad items" in $I'_1(r)$ and $I'_2(r)$ will be offset by the counterparts in $I'_3(r)$ and $I'_4(r)$,
i.e.,
\begin{align}
R_1^5+R_3^3=0, \quad R_1^4+R_3^2=0,\qquad   R_2^6+R_4^3=0, \quad   R_2^5+R_4^2=0,
\end{align}
we finally obtain
\begin{align}\label{I'}
I'(r)&=\frac{2\alpha+n+1}{r}I(r)+\frac{4(\alpha+1)}{r}\int_{B_r(z_0)}\Big(\nabla u\cdot(z-z_0)\Big)^2(r^2-|z-z_0|^2)^{\alpha}dz\nonumber\\[2mm]
&+\frac{4(\alpha+1)}{r}\int_{B_r(z_0)}\Big(\nabla w\cdot(z-z_0)\Big)^2(r^2-|z-z_0|^2)^{\alpha}dz\nonumber\\[2mm]
&+\frac{2}{r}I_3(r)+\frac{2}{r}I_{4}(r)+\frac{2}{r}I_{5}(r)\nonumber\\[2mm]
&+R_{1}^{1}+R_{1}^{2}+R_{1}^{3}+R_{2}^{1}+R_{2}^{2}+R_{2}^{3}+R_{2}^{4}
+R_{3}^{1}+R_{4}^{1}\nonumber\\[2mm]
&+R_{5}^{1}+R_{5}^{2}+R_{5}^{3}+R_{5}^{4}+R_{5}^{5}+R_{5}^{6}+R_{5}^{7}+R_{5}^{8}+R_{5}^{9}+R_{5}^{10}.
\end{align}
Next, we shall estimate the  ``bad'' terms as $\frac{2}{r}I_5(r), R_{1}^1,\cdots,R_5^{10}$ in $I'(r)$. Firstly,
\begin{align*}
R_{1}^1+R_5^7=-\frac{1}{2(\alpha+2)r}\int_{B_r(z_0)}w^2(r^2-|z-z_0|^2)^{\alpha+2}dz,
\end{align*}
which yields
\begin{align}\label{est-R}
\left|R_{1}^1+R_5^7\right|
\leq \frac{1}{2(\alpha+2)}r^{3} H(r)\leq  r^{3} H(r),
\end{align}
thanks to $\alpha>0$.
Similarly,
\begin{align*}
R_{1}^2+R_3^1=-\frac{1}{2(\alpha+2)r}\int_{B_r(z_0)}\lambda uw(r^2-|z-z_0|^2)^{\alpha+2}dz,
\end{align*}
Then, by using Cauchy's inequality, one has
\begin{align}
\left|R_{1}^2+R_3^{1}\right|
&\leq \frac{\sqrt{3}}{6} \frac{1}{\alpha+2}r\int_{B_r(z_0)}\left(\frac{1}{2}\lambda u^2+\frac{3}{2}\lambda w^2\right)(r^2-|z-z_0|^2)^{\alpha+1}dz\nonumber\\[2mm]
&\leq  r \Big(I_3(r)+I_{4}(r)\Big).
\end{align}
The term $R_{1}^3$ can be controlled by
\begin{align}
\left|R_{1}^3\right|&\leq  \frac{1}{2(\alpha+2)}r\int_{B_r(z_0)}\left(|\nabla u|^2+|\nabla w|^2\right)(r^2-|z-z_0|^2)^{\alpha+1}dz\nonumber\\[2mm]
&\leq  r \Big(I_1(r)+I_{2}(r)\Big).
\end{align}
Next,  in order to estimate  terms such as $R_{2}^1$, $R_{2}^3$, $R_{2}^2$, $R_{2}^4$, $R_5^4$ and $R_5^{10}$ by $\|\nabla V\|_{L^{\infty}}H(r)$ and $I(r)$, we shall take  $\alpha=\|\nabla V\|_{L^{\infty}}>0$. Thus,
recalling (\ref{V}), we obtain
\begin{align}
\left|R_{2}^1+R_5^2\right|&=\left|\frac{1}{2(\alpha+2)r}\int_{B_r(z_0)}\left(V-\frac{1}{4}\lambda^2\right)^2 u^2(r^2-|z-z_0|^2)^{\alpha+2}dz\right|\nonumber\\[2mm]
&\leq \frac{1}{2(\alpha+2)}r^{3}\|\nabla V\|_{L^{\infty}}^2\int_{B_r(z_0)}u^2(r^2-|z-z_0|^2)^{\alpha}dz\nonumber\\[2mm]
&\leq Cr^{3}\|\nabla V\|_{L^{\infty}} H(r),
\end{align}
and
\begin{align}
\left|R_{2}^2+R_4^1\right|&=\left|\frac{3}{2(\alpha+2)r}\int_{B_r(z_0)}\lambda\left(V-\frac{1}{4}\lambda^2\right) uw(r^2-|z-z_0|^2)^{\alpha+2}dz\right|\nonumber\\[2mm]
&\leq  \frac{ \sqrt{3}}{2(\alpha+2)}r\|\nabla V\|_{L^{\infty}}\int_{B_r(z_0)}\left(\frac{1}{2}\lambda u^2+\frac{3}{2}\lambda w^2\right)(r^2-|z-z_0|^2)^{\alpha+1}dz\nonumber\\[2mm]
&\leq Cr\Big(I_3(r)+I_{4}(r)\Big).
\end{align}
In a similar way, it holds
\begin{align}
\left|R_{2}^3\right|&\leq \frac{1}{2(\alpha+2)r}\|\nabla V\|_{L^{\infty}}\int_{B_r(z_0)}\left(u^2+|\nabla w|^2\right)(r^2-|z-z_0|^2)^{\alpha+2}dz\nonumber\\[2mm]
&\leq  C\Big(r^{3} H(r)+rI_{2}(r)\Big),
\end{align}
and
\begin{align}
\left|R_{2}^4\right|
&\leq  \frac{1}{2(\alpha+2)}r\|\nabla V\|_{L^{\infty}} \Big(I_{1}(r)+I_{2}(r)\Big)\nonumber\\[2mm]
&\leq Cr\Big(I_{1}(r)+I_{2}(r)\Big).
\end{align}
Terms $R_5^3,R_5^4,R_5^5, R_5^6, R_5^7, R_5^{10}$ can be controlled as
\begin{align}
&\left|R_5^1\right|\leq \frac{1}{2}\|\nabla V\|_{L^{\infty}}r^2 \int_{B_r(z_0)}\left(u^2+w^2\right)(r^2-|z-z_0|^2)^{\alpha}dz\nonumber\\[2mm]
&\quad\quad\leq \frac{r^2}{2}\|\nabla V\|_{L^{\infty}}H(r),\nonumber\\[2mm]
&\left|R_5^3\right|\leq \frac{r^3}{2(\alpha+2)}\|\nabla V\|_{L^{\infty}}H(r)\leq Cr^3H(r),\nonumber\\[2mm]
&\left|R_5^4\right|\leq \frac{1}{2(\alpha+2)}r\|\nabla V\|_{L^{\infty}}\int_{B_r(z_0)}\lambda \left(u^2+w^2\right)(r^2-|z-z_0|^2)^{\alpha+1}dz\nonumber\\[2mm]
&\quad\quad\leq Cr\Big(I_3(r)+I_4(r)\Big),\nonumber\\[2mm]
&\left|R_5^5\right|\leq \frac{1}{2(\alpha+2)}r\int_{B_r(z_0)}\lambda (u^2+w^2)(r^2-|z-z_0|^2)^{\alpha+1}dz\nonumber\\[2mm]
&\quad\quad\leq r \Big(I_3(r)+I_4(r)\Big),\nonumber\\[2mm]
&\left|R_5^6\right|\leq \frac{1}{2(\alpha+2)}r^3\|\nabla V\|_{L^{\infty}}H(r)\leq Cr^3H(r),\nonumber\\[2mm]
&\left|R_5^7\right|\leq \frac{1}{2(\alpha+2)}r^3 H(r)\leq r^3H(r),
\end{align}
and
\begin{align}
\left|R_5^{10}\right|&\leq \frac{r^3}{4(\alpha+2)}\|\nabla V\|_{L^{\infty}} H(r)+\frac{r}{4(\alpha+2)}\|\nabla V\|_{L^{\infty}}\Big(I_1(r)+I_2(r)\Big)\nonumber\\[2mm]
&\leq C\Big(r^3H(r)+r(I_1(r)+I_2(r))\Big).
\end{align}
Recalling
\begin{align*}
I_5(r)=\int_{B_r(z_0)}\left(1+V-\frac{1}{4}\lambda^2\right)uw(r^2-|z-z_0|^2)^{\alpha+1}dz,
\end{align*}
and using the Cauchy inequality and (\ref{V}) again, we derive that
 \begin{align}\label{est-I5}
\left|\frac{2}{r}I_5(r)\right|
&\leq Cr\int_{B_r(z_0)}(u^2+w^2)(r^2-|z-z_0|^2)^{\alpha}dz\nonumber\\[2mm]
&+C\|\nabla V\|_{L^{\infty}}r\int_{B_r(z_0)}(u^2+w^2)(r^2-|z-z_0|^2)^{\alpha}dz\nonumber\\[2mm]
&\leq C(1+\|\nabla V\|_{L^{\infty}})rH(r),
\end{align}
and
\begin{align*}
\left|I_5(r)\right|
\leq C(1+\|\nabla V\|_{L^{\infty}})r^2H(r).
\end{align*}
Moreover,  recalling the expression of $I(r)$ in (\ref{I-0}), and noting that $I_i(r)\geq0$ $(i=1,\cdots, 4)$, it is easy to see that
\begin{align}\label{1+2}
I_1(r)+I_2(r)\leq I(r)-I_5(r)\leq I(r)+C\left(1+\|V\|_{L^{\infty}}\right)H(r),
\end{align}
and
\begin{align}\label{3+4}
I_3(r)+I_4(r)\leq I(r)-I_5(r)\leq I(r)+C\left(1+\|V\|_{L^{\infty}}\right)H(r).
\end{align}
Summing the estimates (\ref{est-R})-(\ref{est-I5})  and using (\ref{1+2}), (\ref{3+4}),   we finally obtain  the estimates of $R_i^{j}$ in (\ref{I'}), that is,
\begin{align*}
\left|R_i^{j}\right|
\leq  C\Big(1+\|V\|_{L^{\infty}}\Big)H(r)+CI(r), \quad \mbox{for}\quad 0<r<1.
\end{align*}
Therefore, $I'(r)$ can be estimated as
\begin{align}\label{fin-2}
I'(r)
&\geq \frac{2\alpha+n+1}{r}I(r)+\frac{4(\alpha+1)}{r}\int_{B_r(z_0)}\Big(\nabla u\cdot(z-z_0)\Big)^2(r^2-|z-z_0|^2)^{\alpha}dz\nonumber\\[2mm]
&+\frac{4(\alpha+1)}{r}\int_{B_r(z_0)}\Big(\nabla w\cdot(z-z_0)\Big)^2(r^2-|z-z_0|^2)^{\alpha}dz\nonumber\\[2mm]
&- C\Big(1+\|V\|_{L^{\infty}}\Big)H(r)-CI(r).
\end{align}

{\it Step 3.}\ Estimate $N'(r)$. Combining $H'(r)$ in (\ref{fin-1}) and $I'(r)$ in (\ref{fin-2}), one has
\begin{align*}\label{F}
N'(r)&=\frac{I'(r)H(r)-H'(r)I(r)}{H^2(r)}\nonumber\\[2mm]
&\geq
\frac{1}{H^2(r)}\Big\{
\Big[\frac{4(\alpha+1)}{r}\int_{B_r(z_0)}\Big(\nabla u\cdot(z-z_0)\Big)^2(r^2-|z-z_0|^2)^{\alpha}dz\nonumber\\[2mm]
&+\frac{4(\alpha+1)}{r}\int_{B_r(z_0)}\Big(\nabla w\cdot(z-z_0)\Big)^2(r^2-|z-z_0|^2)^{\alpha}dz\Big]\cdot H(r)\nonumber\\[2mm]
&-\frac{1}{(\alpha+1)r}I^2(r)
-C\Big(\|\nabla V\|_{L^{\infty}}+1\Big)H^2(r)-CI(r)H(r)\Big\}\nonumber\\[2mm]
&\geq\frac{1}{H^2(r)}\left\{-C\Big(\|\nabla V\|_{L^{\infty}}+1\Big)H^2(r)-CI(r)H(r)\right\}\nonumber\\[2mm]
&\geq -C\Big(\|\nabla V\|_{L^{\infty}}+1\Big)-CN(r),
\end{align*}
where we have used the Cauchy inequality in the following form
\begin{align*}
&\Big\{\frac{4(\alpha+1)}{r}\int_{B_r(z_0)}\Big(\nabla u\cdot(z-z_0)\Big)^2(r^2-|z-z_0|^2)^{\alpha}dz\\[2mm]
&+\frac{4(\alpha+1)}{r}\int_{B_r(z_0)}\Big(\nabla w\cdot(z-z_0)\Big)^2(r^2-|z-z_0|^2)^{\alpha}dz\Big\}\cdot H(r)
-\frac{1}{(\alpha+1)r}I^2(r)\geq 0,
\end{align*}
thanks to the definitions of $H(r), I(r)$ in (\ref{H}) and (\ref{I}).
This implies the  almost monotonicity property of $N(r)$, this completes the proof of the lemma.
\end{proof}

\begin{remark}
In the following proof, we always  take $\alpha=\|\nabla V\|_{L^{\infty}}$.
\end{remark}

Next, we are going to establish some doubling estimates of $H(r)$ as well as  the following $L^2$-integral without weight.
Let
\begin{equation*}
h(z_0,r)=\int_{B_r(z_0)}(u^2+w^2)dz.
\end{equation*}
It is easy to check that
\begin{equation}\label{h-H-1}
H(z_0,r)\leq r^{2\alpha}h(z_0,r),
\end{equation}
and
\begin{equation}\label{h-H-2}
h(z_0,r)\leq \frac{H(z_0,\rho)}{(\rho^2-r^2)^{\alpha}},\quad 0<r<\rho<1.
\end{equation}

With the help of the almost monotonicity of $N(r)$ in Lemma \ref{keylemma}, we are able to establish the following  doubling estimates.
\begin{lemma}\label{lem2}
For $0<r_1<r_2<1$, it holds
\begin{align}
H(z_0,r_2)\label{dou-1}
\leq\left(\frac{r_2}{r_1}\right)^{(2\alpha+n+1)+\frac{C\left(N(z_0,r_2)+\|\nabla V\|_{L^{\infty}}+1\right)}{\alpha+1}}H(z_0,r_1),
\end{align}
and
\begin{align}\label{dou-2}
H(z_0,r_2)
\geq\left(\frac{r_2}{r_1}\right)^{(2\alpha+n+1)+\frac{C^{-1}N(z_0,r_1)-\|\nabla V\|_{L^{\infty}}-1}{\alpha+1}}H(z_0,r_1).
\end{align}
Furthermore, for $0<r_1<r_2<2r_2<1$, it holds
\begin{align}\label{dou-3}
h(z_0,r_2)
\leq\left(\frac{4}{3}\right)^{\alpha}\left(\frac{2r_2}{r_1}\right)^{n+1+\frac{C(N(z_0, 2r_2)+\|\nabla V\|_{L^{\infty}}+1)}{\alpha+1}}h(z_0,r_1),
\end{align}
and
\begin{align}\label{dou-4}
h(z_0,r_2)\geq \left(\frac{3}{4}\right)^{\alpha}\left(\frac{r_2}{2r_1}\right)^{n+1+\frac{C^{-1}N(z_0,2r_1)-\|\nabla V\|_{L^{\infty}}-1}{\alpha+1}}h(z_0,r_1),
\end{align}
where constant $C$ depends only on $n$.
\end{lemma}
\begin{proof}
Lemma \ref{keylemma} implies, for any $0<r_1<r<r_2<1$, that
\begin{align}\label{almost}
N(z_0,r)&\leq e^{C(r_2-r_1)}\Big(N(z_0,r_2)+\|\nabla V\|_{L^{\infty}}+1\Big)-\|\nabla V\|_{L^{\infty}}-1\nonumber\\[2mm]
&\leq C \Big(N(z_0,r_2)+\|\nabla V\|_{L^{\infty}}+1\Big),
\end{align}
and
\begin{align}\label{almost2}
N(z_0,r)&\geq e^{-C(r_2-r_1)}\Big(N(z_0,r_1)+\|\nabla V\|_{L^{\infty}}+1\Big)-\|\nabla V\|_{L^{\infty}}-1\nonumber\\[2mm]
&\geq C^{-1}N(z_0,r_1)-\|\nabla V\|_{L^{\infty}}-1.
\end{align}
From the equality $(\ref{fin-1})$, one has
\begin{equation}\label{H-2}
\frac{H'(z_0,r)}{H(z_0,r)}=\frac{2\alpha+n+1}{r}+\frac{N(z_0,r)}{(\alpha+1)r}.
\end{equation}
Then, integrating form $r_1$ to $r_2$ on (\ref{H-2}), and using  (\ref{almost}), we get
\begin{align*}
\log\frac{H(z_0,r_2)}{H(z_0,r_1)}&=\int_{r_1}^{r_2}\left(\frac{2\alpha+n+1}{r}+\frac{N(z_0,r)}{(\alpha+1)r}\right)dr\\[2mm]
&\leq\left\{(2\alpha+n+1)+\frac{C}{\alpha+1}\Big(N(z_0,r_2)+\|\nabla V\|_{L^{\infty}}+1\Big)\right\}\log \frac{r_2}{r_1},
\end{align*}
which yields
\begin{align*}
\frac{H(z_0,r_2)}{H(z_0,r_1)}
&\leq\left(\frac{r_2}{r_1}\right)^{(2\alpha+n+1)+\frac{C(N(z_0,r_2)+\|\nabla V\|_{L^{\infty}}+1)}{\alpha+1}}.
\end{align*}
On the other hand,  (\ref{almost2}) yields
\begin{align*}
\log\frac{H(z_0,r_2)}{H(z_0,r_1)}
\geq\left\{(2\alpha+n+1)+\frac{1}{\alpha+1}\Big(C^{-1}N(z_0,r_1)-\|\nabla V\|_{L^{\infty}}-1\Big)\right\}\log \frac{r_2}{r_1},
\end{align*}
which means
\begin{align*}
\frac{H(z_0,r_2)}{H(z_0,r_1)}
&\geq\left(\frac{r_2}{r_1}\right)^{(2\alpha+n+1)+\frac{C^{-1}N(z_0,r_1)-\|\nabla V\|_{L^{\infty}}-1}{(\alpha+1)}}.
\end{align*}
That is,  (\ref{dou-1}) and (\ref{dou-2}) are proved.

Moreover, according to the relationships with $h(r)$ and $H(r)$, it holds
\begin{align*}
\frac{h(z_0,r_2)}{h(z_0,r_1)}&\leq\frac{(3r_2^2)^{-\alpha}H(z_0,2r_2)}{r_1^{-2\alpha}H(z_0,r_1)}\nonumber\\[2mm]
&\leq\left(\frac{4}{3}\right)^{\alpha}\left(\frac{2r_2}{r_1}\right)^{n+1+\frac{C(N(z_0, 2r_2)+\|\nabla V\|_{L^{\infty}}+1)}{\alpha+1}},
\end{align*}
and
\begin{align*}
\frac{h(z_0,r_2)}{h(z_0,r_1)}&\geq \frac{r_2^{-2\alpha}H(z_0,r_2)}{(3r_1^2)^{-\alpha}H(z_0,2r_1)}\nonumber\\[2mm]
&\geq \left(\frac{3}{4}\right)^{\alpha}\left(\frac{r_2}{2r_1}\right)^{n+1+\frac{C^{-1}N(z_0,2r_1)-\|\nabla V\|_{L^{\infty}}-1}{\alpha+1}}.
\end{align*}
This finishes the proof.
\end{proof}

At the end of this section,
we are going to derive the following ``changing center'' property of the frequency function.
\begin{lemma}\label{lem-change}
For any $z_1\in B_{r/32}(z_0)$, it holds
\begin{align}\label{chang}
N\left(z_1,r/8\right)\leq  C\left\{\|\nabla V\|_{L^{\infty}}+(\alpha+1)^2+N(z_0,9r/16)\right\},
\end{align}
where $C$ is a positive constant depending only on $n$. In particular, for any $z\in B_{1/32}(0)$ and any $\rho<\frac{1}{8}$,
\begin{align}
N(z,\rho)\leq C\left\{\|\nabla V\|_{L^{\infty}}+(\alpha+1)^2+N(0,1)\right\},
\end{align}
where $C$ is a positive constant depending only on $n$.
\end{lemma}
\begin{proof}
By  (\ref{dou-4}), we have
\begin{align}\label{dou-5}
N\left(z_1,2r_1\right)&\leq C\left\{\|\nabla V\|_{L^{\infty}}+1+(\alpha+1)\left(\log\frac{r_2}{2r_1}\right)^{-1}
\log\left(\left(\frac{4}{3}\right)^{\alpha}\cdot\frac{h(z_1,r_2)}{h(z_1,r_1)}\right)\right\}.
\end{align}
Taking $r_1=r/16$, $r_2=r/4$ in (\ref{dou-5}), and noting that $B(z_1,r/4)\subseteq B(z_0,9r/32)$, and $B(z_0,r/32)\subseteq B(z_1,r/16)$  for $z_1\in B(z_0,r/32)$,  we obtain
\begin{align}\label{chang1}
N\left(z_1,r/8\right)&\leq C\left\{\|\nabla V\|_{L^{\infty}}+1+(\alpha+1)\log \left(\frac{4}{3}\right)^{\alpha}+(\alpha+1)
\log\frac{h(z_1,r/4)}{h(z_1,r/16)}\right\}\nonumber\\[2mm]
&\leq  C\left\{\|\nabla V\|_{L^{\infty}}+1+(\alpha+1)\log \left(\frac{4}{3}\right)^{\alpha}+(\alpha+1)
\log\frac{h(z_0,9r/32)}{h(z_0,r/32)}\right\}.
\end{align}
On the other hand,  (\ref{dou-3}) implies
\begin{align}\label{chang2}
\log\frac{h(z_0,9r/32)}{h(z_0,r/32)}&\leq \log \left(\frac{4}{3}\right)^{\alpha}+ \left(n+1+\frac{C(N(z_0,9r/16)+\|\nabla V\|_{L^{\infty}}+1)}{\alpha+1}\right)\log\frac{9}{2}\nonumber\\[2mm]
&\leq C\left(\alpha+1+\frac{N(z_0,9r/16)}{\alpha+1}\right).
\end{align}
Putting (\ref{chang2}) into (\ref{chang1})  implies the desired result (\ref{chang}). Moreover,  by the almost  monotonicity property of $N(r)$ in (\ref{almost}), we obtain (\ref{chang2}). 
\end{proof}
\section{The maximal vanishing order}

In this section, we prove Theorem \ref{thm}. In order to  establish the quantitative relationship between the vanishing order
and the frequency,   we need to  some interior estimates for solutions of (\ref{equ}).

We first prove a  Caccioppoli type inequality  for solutions to (\ref{equ}). In such an inequality, we will focus on how  the coefficient depends on the norms of the potential $V$. Precisely,
\begin{lemma}\label{lem4}
Let $u$ and $w$ be the solutions of
(\ref{equ}), then there exists a positive constant $C$ depending only on $n$ such that
\begin{equation}\label{cap}
\int_{B_{r}(z_0)}w^2dz\leq C(\lambda^2+1) r^{-4}\int_{B_{2r}(z_0)}u^2dz,
\end{equation}for any $B_{2r}(z_0)\subset B_1(0,0).$
\end{lemma}
\begin{proof}
Suppose that $u$ and $w$ are  solutions of (\ref{equ}), then $u$ is a solution of (\ref{equ2}), that is,
\begin{align}\label{test}
\int_{B_{2r}(z_0)}\Delta u\Delta\varphi dz=\int_{B_{2r}(z_0)}\Big(2\lambda \Delta u+ (V-\lambda^2)u\Big)\varphi dz
\end{align}
for any $\varphi\in W^{2,2}_0(B_{2r}(z_0))$.
Let $\eta\in C_0^{\infty}$ be a cut-off function:
\begin{align}\label{cut}
0\leq\eta\leq1;\quad \eta\equiv1 \ \mbox{for} \ |z-z_0|\leq r,\quad \eta=0 \ \mbox{for}\  |z-z_0|\geq 2r;
\quad |\nabla \eta|\leq \frac{C}{r}, \quad |\nabla^2 \eta|\leq \frac{C}{r^2}.
\end{align}
Taking test function $\varphi=u\eta^4$ in  (\ref{test}), we
deduce that
\begin{align*}
\int_{B_{2r}(z_0)}\Delta u \Delta (u\eta^4)dz=\int_{B_{2r}(z_0)} 2\lambda \Delta u u\eta^4 dz+\int_{B_{2r}(z_0)}(V-\lambda^2)u^2\eta^4dz,
\end{align*}
which implies
\begin{align*}
\int_{B_{2r}(z_0)}(\Delta u)^2\eta^4dz&=\int_{B_{2r}(z_0)} 2\lambda \Delta u u\eta^4 dz +\int_{B_{2r}(z_0)}(V-\lambda^2)u^2\eta^4dz
-4\int_{B_{2r}(z_0)}\eta^3u\Delta u\Delta\eta dz\\[2mm]
&-12\int_{B_{2r}(z_0)}\eta^2u\Delta u|\nabla\eta|^2dz
-8\int_{B_{2r}(z_0)}\eta^3\Delta u\nabla u\cdot \nabla\eta dz.
\end{align*}
By  using the Young  inequality and noting $V-\lambda^2<0$,
\begin{align}\label{ineq}
\int_{B_{2r}(z_0)}(\Delta u)^2\eta^4dz&\leq \frac{1}{4}\int_{B_{2r}(z_0)}(\Delta u)^2\eta^4dz+C\left\{\int_{B_{2r}(z_0)}\lambda^2 u^2\eta^4dz
+\int_{B_{2r}(z_0)}\eta^2u^2(\Delta\eta)^2dz\right\}\nonumber\\[2mm]
&+C\left\{\int_{B_{2r}(z_0)}u^2|\nabla\eta|^4dz
+\int_{B_{2r}(z_0)}\eta^2|\nabla u|^2|\nabla\eta|^2dz\right\}.
\end{align}
On the other hand,  integrating by parts, one has
\begin{align*}
\int_{B_{2r}(z_0)}\Delta u (u\eta^2)dz=-\int_{B_{2r}(z_0)} |\nabla u|^2\eta^2dz-2\int_{B_{2r}(z_0)}u\eta \nabla u\cdot\nabla\eta dz,
\end{align*}
Then, by using the  Young inequality and (\ref{cut}), we get
\begin{eqnarray*}
\int_{B_{2r}(z_0)}|\nabla u|^2\eta^2dz\leq \frac{1}{4}  r^2\int_{B_{2r}(z_0)}(\Delta u)^2\eta^4dz+Cr^{-2}\int_{B_{2r}(z_0)}u^2dz.
\end{eqnarray*}
Therefore, using (\ref{cut}) again, we  have
\begin{align}\label{w4}
\int_{B_{2r}(z_0)}\eta^2|\nabla u|^2|\nabla \eta|^2dz
\leq \frac{1}{4} \int_{B_{2r}(z_0)}(\Delta u)^2\eta^4+Cr^{-4}\int_{B_{2r}(z_0)}u^2dz.
\end{align}
Plugging (\ref{w4}) into (\ref{ineq}) and using (\ref{cut}), we obtain
\begin{align*}
\int_{B_{2r}(z_0)}(\Delta u)^2\eta^4dz\leq C r^{-4}\int_{B_{2r}(z_0)}(\lambda^2+1)u^2dz,
\end{align*}
which yields the desired inequality (\ref{cap}) thanks to the equation (\ref{equ}).
\end{proof}

Next, we shall  to  establish some $L^{\infty}$-$L^2$ estimate for the solutions to (\ref{equ}) by using  the classical De Giorgi-Nash-Moser theory of the second-order case.
\begin{lemma}\label{lem3}
Let $u$ be the solution of (\ref{equ}). There exists a positive constant $C$ depending only on $n$
such that
\begin{align}
\|u\|_{L^{\infty}(B_{r}(z_0))}\leq C\left\{ \lambda^2+1+\|\nabla V\|^{\frac{n+2}{4}}_{L^{\infty}} \right\}r^{-\frac{n+1}{2}} \|u\|_{L^{2}(B_{2r}(z_0))},
\end{align}
for any $B_{2r}(z_0)\subset B_1(0,0)$.
\end{lemma}
\begin{proof}
Since $u$ is a solution of the first equation in (\ref{equ}), that is,
\begin{align}\label{weak1}
-\int_{B_1}\nabla u\cdot\nabla \phi dz-\frac{1}{2}\int_{B_1}\lambda u\phi dz=\int_{B_1}w\phi dz, \quad \mbox{for any} \quad  \phi\in H^{1}_0(B_1).
\end{align}
Let $\bar{u}=(u-k)^{+}$ for $k\geq0$ and $\xi\in C_0^1(B_1)$. Set $\phi=\bar{u}\xi^2$ as the test function in (\ref{weak1}) and  note $\lambda>0$. Then, by using the H\"{o}lder inequality we have
\begin{align}
\int_{\{u>k\}}|\nabla (\bar{u}\xi)|^2\leq C\int_{\{u>k\}}\bar{u}^2|\nabla\xi|^2+\int_{\{u>k\}}|w|\bar{u}\xi^2.
\end{align}
It is worth noting that  $C$ is a  constant independent of $\lambda$.
Then,  according  to the  De Giorgi's  iterative process (see e.g. Theorem 4.1 in \cite{book}), we can arrive at the following estimate
\begin{align}\label{est-u}
\|u\|_{L^{\infty}(B_{1/8})}\leq C\left\{\|u\|_{L^2(B_{1/4})}+\|w\|_{L^{q}(B_{1/4})}\right\},  \quad \mbox{for}\quad  q>\frac{n+1}{2},
\end{align}
where $C=C(n,q)$ is a positive constant depending only on $n$ and $q$.

In a similar way,  we write the second equation in (\ref{equ}) as an inhomogeneous equation
 \begin{align*}
 \Delta w-\frac{3}{2}\lambda w=f\equiv (V-\frac{1}{4}\lambda^2)u,
 \end{align*}
where $\lambda>0$.
Similarly,  applying the  De Giorgi's  iterative argument, for $s>\frac{n+1}{2}$ there holds
\begin{align}\label{est-w}
\|w\|_{L^{q}(B_{1/4})}
\leq C(q,n)\|w\|_{L^{\infty}(B_{1/4})}
\leq C(n,q,s)\left\{\|w\|_{L^2(B_{1/2})}+\left\|\left(V-\frac{1}{4}\lambda^2\right)u\right\|_{L^{s}(B_{1/2})}\right\}.
\end{align}
Taking $s=\frac{n+2}{2}$, so $s\geq 2$ thanks to $n\geq 2$, using Young's inequality, the last term on above can be estimated as
\begin{align}\label{s}
\left\|(V-\frac{1}{4}\lambda^2)u\right\|_{L^{s}(B_{1/2})}&\leq \left\|V-\frac{1}{4}\lambda^2\right\|_{L^{\infty}}\left\|u\right\|_{L^{\frac{n+2}{2}}(B_{1/2})}\nonumber\\[2mm]
&\leq \|\nabla V\|_{L^{\infty}}\|u\|^{\frac{n-2}{n+2}}_{L^{\infty}(B_{1/2})}\|u\|^{\frac{4}{n+2}}_{L^{2}(B_{1/2})}\nonumber\\[2mm]
&\leq \varepsilon \|u\|_{L^{\infty}(B_{1/2})}+C(\varepsilon)\|\nabla V\|^{\frac{n+2}{4}}_{L^{\infty}}\|u\|_{L^{2}(B_{1/2})},
\end{align}
for any small $\varepsilon>0$. Putting  (\ref{est-w}) and (\ref{s}) into (\ref{est-u}), and using the  Caccioppoli type inequality (\ref{cap}), we infer that
\begin{align}
\|u\|_{L^{\infty}(B_{1/8})}\leq  \frac{1}{2} \|u\|_{L^{\infty}(B_{1/2})}+C\Big(\|\nabla V\|^{\frac{n
+2}{4}}_{L^{\infty}}+\lambda^2+1\Big)\|u\|_{L^{2}(B_1)}.
\end{align}
We apply the following lemma  and a stand  scaling argument to get
\begin{align*}
\|u\|_{L^{\infty}(B_{r})}\leq  C\left\{\|\nabla V\|^{\frac{n+2}{4}}_{L^{\infty}}+\lambda^2+1\right\}r^{-\frac{n+1}{2}}\|u\|_{L^{2}(B_{2r})}.
\end{align*}
This completes the proof.
\end{proof}

Over the proof of Lemma \ref{lem3}, we need the following lemma ( see e.g. Lemma 4.3 in \cite{book}).
\begin{lemma}
Let $f(t)\geq0$ be bounded in $[\tau_0,\tau_1]$ with $\tau_0\geq 0$. Suppose for
$\tau_0\leq t<s\leq \tau_1$ we have
\begin{align*}
f(t)\leq \theta f(s)+\frac{A}{(s-t)^{\alpha}}+B,
\end{align*}
for some $\theta\in [0,1)$. Then for any $\tau_0\leq t<s\leq \tau_1$ there holds
\begin{align*}
f(t)\leq C(\alpha,\theta)\left\{\frac{A}{(s-t)^{\alpha}}+B\right\}.
\end{align*}
\end{lemma}

Now we are ready to prove the quantitative unique continuation property.

{\bf {Proof of Theorem \ref{thm}}}\
Denote the maximal vanishing order of $u$ at $z_0$ as $\mathcal{M}(u,z_0)$. Now, we shall establish a quantitative relationship
between $\mathcal{M}(u,z_0)$ and $N(z_0,r)$.
By using Lemma \ref{lem4} and (\ref{dou-3}) with $r_1=\frac{r}{2}<r_2$, we obtain
\begin{align*}
\frac{\|u\|^2_{L^{2}(B_r(z_0))}}{r^{n+1}}&\geq \frac{C^{-1}\left(\lambda^{2}+1\right)^{-1}r^4\left(\|u\|_{L^{2}(B_{r/2}(z_0))}+\|w\|_{L^{2}(B_{r/2}(z_0))}\right)}{r^{n+1}}\nonumber\\[2mm]
&\geq C^{-1}\left(\lambda^{2}+1\right)^{-1}\frac{h(z_0,r/2)}{r^{n+1-4}}\\[2mm]
&\geq C^{-1}\left(\lambda^{2}+1\right)^{-1}\left(\frac{3}{4}\right)^{\alpha}\left(\frac{1}{4r_2}\right)^{n+1+\frac{C(N(z_0,2r_2)+\|\nabla V\|_{L^{\infty}}+1)}{\alpha+1}}h(z_0,r_2)
r^{\frac{C(N(z_0,2r_2)+\|\nabla V\|_{L^{\infty}}+1)}{\alpha+1}+4}.
\end{align*}
According to the definition of vanishing order (\ref{def}),  the above estimate yields
\begin{align}\label{4}
\mathcal{M}(u,z_0)\leq\frac{C(N(z_0,2r_2)+\|\nabla V\|_{L^{\infty}}+1)}{\alpha+1}+2.
\end{align}
Now, we shall estimate $N\Big((0,0),\frac{1}{4}\Big)$. Since
\begin{align*}
u(x,t)=e^{\sqrt{\lambda}t}u(x),
\end{align*}
and  $|u(0)|\geq 1$, $\|u(x)\|_{L^{\infty}\left(B_{1}(0)\right)}\leq C_0$.
So,
\begin{align}\label{assump}
|u(0,0)|\geq1, \quad \|u(x,t)\|_{L^{\infty}\left(B_{1}(0,0)\right)}\leq C_0e^{\sqrt{\lambda}}.
\end{align}
From (\ref{dou-2}) with  $r_1=\frac{1}{4}, r_2=\frac{1}{2}$, it holds
\begin{align}\label{2}
N\Big((0,0),\frac{1}{4}\Big)\leq C\left\{ \|\nabla V\|_{L^{\infty}}+1+\Big(\alpha+1\Big)\frac{\log\frac{H\left((0,0),\frac{1}{2}\right)}{H\left((0,0),\frac{1}{4}\right)}}{\log 2}\right\}.
\end{align}
Next, we shall estimate $\log\frac{H\left((0,0),\frac{1}{2}\right)}{H\left((0,0),\frac{1}{4}\right)}$ by virtue of the assumptions (\ref{assump}).
By Lemma \ref{lem4}, we infer that
\begin{align*}
H((0,0),r_2)&\leq r_2^{2\alpha}\int_{B_{r_2}(0,0)}(u^2+w^2)dz\nonumber\\[2mm]
&\leq C r_{2}^{2\alpha}\left\{\int_{B_{r_2}(0,0)}u^2dz+(\lambda ^2+1)r_2^{-4}\int_{B_{2r_2}(0,0)}u^2dz\right\}\nonumber\\[2mm]
&\leq C(\lambda^2+1)r_{2}^{2\alpha-4}\int_{B_{2r_2}(0,0)}u^2dz\nonumber\\[2mm]
&\leq C(\lambda^2+1)r_{2}^{2\alpha-4+n+1}\|u\|^2_{L^{\infty}({B_{2r_2}(0,0)})}\nonumber\\[2mm]
&\leq C^2_0C(\lambda^2+1)r_{2}^{2\alpha-4+n+1}e^{2\sqrt{\lambda}},
\end{align*}
which implies
\begin{align}\label{1}
H((0,0),\frac{1}{2})
\leq C^2_0C(\lambda^2+1)\left(\frac{1}{2}\right)^{2\alpha-4+n+1}e^{2\sqrt{\lambda}}.
\end{align}
On the other hand,  applying  Lemma \ref{lem3}, one has
\begin{align*}
H((0,0),r_1)&\geq \left(\frac{3}{4}r_1^2\right)^{\alpha}\int_{B_{\frac{1}{2}r_1}(0,0)}(u^2+w^2)dz\nonumber\\[2mm]
&\geq C^{-1}\left(\frac{3}{4}r_1^2\right)^{\alpha} r_1^{n+1}\left(\lambda^2+\|\nabla V\|^{\frac{n+1}{4}}_{L^{\infty}}+1\right)^{-2}\|u\|^2_{L^{\infty}({B_{\frac{1}{4}r_1}(0,0)})},
\end{align*}
so,
\begin{align*}
H((0,0),\frac{1}{4})
\geq C^{-1}\left(\frac{3}{64}\right)^{\alpha} \left(\frac{1}{4}\right)^{n+1}\left(\lambda^2+\|\nabla V\|^{\frac{n+1}{4}}_{L^{\infty}}+1\right)^{-2}.
\end{align*}
which together with (\ref{1}) implies that
\begin{align*}
\frac{\log\frac{H((0,0),\frac{1}{2})}{H((0,0),\frac{1}{4})}}{\log2}\leq C\left(\sqrt{\lambda}+\alpha+1\right).
\end{align*}
Then, submitting the above into (\ref{2}), we get
\begin{align}\label{3}
N\Big((0,0),\frac{1}{4}\Big)\leq C(\alpha+1)\left(\sqrt{\lambda}+\alpha+1\right).
\end{align}
Moreover, using the almost monotonicity property of $N(r)$ in (\ref{almost}),  for any $\rho<\frac{1}{4}$, it holds
\begin{align}\label{6}
N\Big((0,0),\rho\Big)\leq C(\alpha+1)\left(\sqrt{\lambda}+\alpha+1\right).
\end{align}
Then, (\ref{4}) and (\ref{3}) yields that the maximal vanishing order of $u$ at $(0,0)$ is
\begin{align*}
\mathcal{M}\left(u,(0,0)\right)&\leq C\left(\sqrt{\lambda}+\alpha+1\right).
\end{align*}
Furthermore, using Lemma \ref{lem-change}, (\ref{4}) and (\ref{6})  again, for any $z_1\in B_{r/32}(0,0)$,
\begin{align*}
\mathcal{M}(u,z_1)&\leq\frac{C(N(z_1,2r_2)+\|\nabla V\|_{L^{\infty}}+1)}{\alpha+1}+2\\[2mm]
&\leq \frac{C(N(0,9r_2)+(\alpha+1)^2+\|\nabla V\|_{L^{\infty}}+1)}{\alpha+1}+2\\[2mm]
&\leq C\left(\sqrt{\lambda}+\alpha+1\right).
\end{align*}
Continuing  the above process, for any $z_2\in B_{\frac{r}{32}}(z_1)$,
\begin{align*}
\mathcal{M}(u,z_2)&\leq\frac{C(N(z_2,2r_2)+\|\nabla V\|_{L^{\infty}}+1)}{\alpha+1}+2\\[2mm]
&\leq \frac{C(N(z_1,9r_2)+(\alpha+1)^2+\|\nabla V\|_{L^{\infty}}+1)}{\alpha+1}+2\\[2mm]
&\leq  \frac{C(N(0,81r_2/2)+(\alpha+1)^2+\|\nabla V\|_{L^{\infty}}+1)}{\alpha+1}+2\\[2mm]
&\leq C\left(\sqrt{\lambda}+\alpha+1\right).
\end{align*}
Therefore,  we fix the small value of
$r$, for instance, let $r=1/100$.
After a finite number of iterations, we can show that  the maximal vanishing order of $u(x,t)$ at any point
$z\in B_{1}(0,0)$ is $ C\left(\sqrt{\lambda}+\alpha+1\right)$. Back to the solutions $u(x)$ of our original equation (\ref{equ1}), recalling $\sqrt{\lambda}=\sqrt{2}\|V\|^{1/4}_{L^{\infty}}$ and $\alpha=\|\nabla V\|_{L^{\infty}}$, we prove that the vanishing order of  $u$ at any point $x\in B_1(0)$ is not large than $C\left(\|V\|^{\frac{1}{4}}_{L^{\infty}}+\|\nabla V\|_{L^{\infty}}+1\right)$.
This finishes the proof.
\qed

\section*{Funding}
This work was supported by the  National Natural Science Foundation of China  under Grant No.12071219 and  No.11971229.

\end{document}